\documentclass[11pt]{amsart}
\usepackage{amssymb}
\usepackage{ textcomp }
\newtheorem{thm}{Theorem}

\newtheorem{cor}[thm]{Corollary}

\newcommand{\al}{\alpha}

\newcommand{\de}{\delta}

\newcommand{\mrig}{\mathrel{-\!\!\!\!\!\rightarrow}}
\newcommand{\Rig}{\Rightarrow}

\newcommand{\bcdw}{\mathbin{\boldsymbol\cdot}}

\newcommand{\seteq}{\mathrel{\mbox{\,\textup{:}\!}=\nolinebreak }\,}

\newcommand{\sbA}{{\boldsymbol{A}}}
\newcommand{\sbB}{{\boldsymbol{B}}}
\newcommand{\sbC}{{\boldsymbol{C}}}

\newcommand{\sbT}{{\boldsymbol{T}}}

\newcommand{\ld}{{\backslash}}

\newcommand{\ov}{\overline}

\newcommand{\monus}{\mbox{$\frac{\:{\bf \displaystyle .}\:}{\mbox{}}$}}

\newcommand{\leibniz}{\boldsymbol{\varOmega}}

\newtheorem{theorem}{Theorem}[section]
\newtheorem{lemma}[theorem]{Lemma}
\newtheorem{corollary}[theorem]{Corollary}

\newtheoremstyle{inproof}
{3pt}
{3pt}
{\normalfont}
{}
{\bfseries}
{:}
{.5em}
{}

\theoremstyle{inproof}

\theoremstyle{definition}
\newtheorem{definition}[theorem]{Definition}

\newtheorem{remark}[theorem]{Remark}

\DeclareMathOperator{\Con}{\mathit{Con}}

\DeclareMathOperator{\Sg}{Sg}
\DeclareMathOperator{\Cg}{\Theta}

\providecommand{\monus}{
	\mathbin{
		\vphantom{+}
		\text{
			\mathsurround=0pt 
			\ooalign{
				\noalign{\kern-.35ex}
				\hidewidth$\smash{\cdot}$\hidewidth\cr 
				\noalign{\kern.35ex}
				$-$\cr 
			}%
		}%
	}%
}
\newcommand{\alg}[1]{{\boldsymbol{#1}}}
\newcommand{\spc}[1]{{\mathcal{#1}}}

\newcommand{\cl}[1]{{\mathsf{#1}}}

\DeclareMathOperator{\Hop}{\mathbb{H}}
\DeclareMathOperator{\Sop}{\mathbb{S}}
\DeclareMathOperator{\Pop}{\mathbb{P}}
\DeclareMathOperator{\Puop}{\mathbb{P}_\mathbb{U}}

\DeclareMathOperator{\Iop}{\mathbb{I}}

\DeclareMathOperator*{\amper}{\&}
\DeclareMathOperator*{\bigamper}{\mbox{\larger[2]$\amper$}}

\makeatletter
\providecommand*{\Dashv}{%
	\mathrel{%
		\mathpalette\@Dashv\vDash
	}%
}
\newcommand*{\@Dashv}[2]{%
	\reflectbox{$\m@th#1#2$}%
}
\makeatother
\makeatletter

\begin{document}
\title[Epimorphisms, definability and cardinalities]{Epimorphisms, definability and cardinalities}
\author{T.\ Moraschini}
\address{Institute of Computer Science, Academy of Sciences of the Czech Republic, Pod Vod\'{a}renskou v\v{e}\v{z}\'{i} 2, 182 07 Prague 8, Czech Republic.}
\email{moraschini@cs.cas.cz}
\author{J.G.\ Raftery}
\address{Department of Mathematics and Applied Mathematics,
 University of Pretoria,
 Private Bag X20, Hatfield,
 Pretoria 0028, South Africa}
\email{{james.raftery@up.ac.za}}
\author{J.J.\ Wannenburg}
\address{Department of Mathematics and Applied Mathematics,
 University of Pretoria,
 Private Bag X20, Hatfield,
 Pretoria 0028, and DST-NRF Centre of Excellence in Mathematical and Statistical Sciences (CoE-MaSS), South Africa}
\email{{jamie.wannenburg@up.ac.za}}
\keywords{Epimorphism, prevariety, quasivariety, Beth definability, algebraizable logic, equivalential logic.
\vspace{1mm}
\\{\vspace{1mm}}
}
%
\thanks{This work received funding from the European Union's Horizon 2020 research and innovation programme under the Marie Sklodowska-Curie grant
agreement No~689176 (project ``Syntax Meets Semantics: Methods, Interactions, and Connections in Substructural logics"). The first author was also supported by
the project GA17-04630S of the Czech Science Foundation (GA\v{C}R).
The second author was supported in part by the National Research Foundation of South Africa (UID 85407).
The third author was supported by the DST-NRF Centre of Excellence in Mathematical and Statistical Sciences (CoE-MaSS), South Africa.
Opinions expressed and conclusions arrived at are those of the authors and are not necessarily to be attributed to the CoE-MaSS}

\begin{abstract}
We
characterize, in syntactic terms,
the ranges of epimorphisms
in an arbitrary class of similar first-order structures (as opposed to an elementary class).  This allows us
to strengthen a result of Bacsich,
as follows: in any prevariety having at most $\mathfrak{s}$ non-logical symbols and an
axiomatization requiring at most $\mathfrak{m}$ variables, if the epimorphisms into
structures with at most
$\mathfrak{m}+\mathfrak{s}+\aleph_0$
elements are surjective, then so are all of the epimorphisms.
Using these facts, we formulate and prove manageable `bridge theorems', matching the
surjectivity of \emph{all} epimorphisms in the algebraic counterpart of a logic $\,\vdash$ with
suitable infinitary definability properties of $\,\vdash$, while \emph{not} making the standard
but awkward assumption that $\,\vdash$ comes furnished with a proper class of variables.
\end{abstract}


\maketitle

\makeatletter
\renewcommand{\labelenumi}{\text{(\theenumi)}}
\renewcommand{\theenumi}{\roman{enumi}}
\renewcommand{\theenumii}{\roman{enumii}}
\renewcommand{\labelenumii}{\text{(\theenumii)}}
\renewcommand{\p@enumii}{\theenumi(\theenumii)}
\makeatother

{\allowdisplaybreaks

\section{Introduction}\label{sec:introduction}

`Bridge theorems' of abstract algebraic logic
\cite{BP89,Cze01,Fon16,FJP03}
have the form
\[
\textup{$\,\vdash$ has logical property $P$ iff $\cl{K}$ has algebraic property $Q$,}
\]
where $\,\vdash$ is an algebraizable logic and $\cl{K}$ is its algebraic counterpart---in which case
$\cl{K}$ is a prevariety, at least.  Examples include connections between Beth definability
properties and the surjectivity of epimorphisms.
Roughly speaking,
Beth properties ask that,
whenever a
set $\Gamma$ of formal assertions about
$\vec{x},\vec{z}$
defines $\vec{z}$ implicitly
in terms of $\vec{x}$,
then it does so explicitly as well.  (Greater precision will be offered in
Section~\ref{sec:beth definability properties}.)

In fact, there are bridge
theorems characterizing several different Beth-style properties by
demands that various \emph{kinds} of epimorphism be surjective.
The first of these was proved by N\'{e}meti; see \cite[Thm.\,5.6.10]{HMT7185}.  Others were
provided in \cite{Hoo00,Hoo01,BH06,Bud08}.  Concrete antecedents involving special families of logics can
be found in Maksimova's work, e.g., see \cite{GM05,Mak00,Mak03}.

Blok and Hoogland \cite{BH06} showed that the straightforward \emph{ES property}---i.e., the demand
that \emph{all} epimorphisms
in $\cl{K}$ be surjective---corresponds to an \emph{infinite} version of the Beth property,
where no cardinal is assumed to bound the lengths of the sequences $\vec{x},\vec{z}$, nor the size of
$\Gamma$.  When testing for implicit definability, we need to substitute expressions for the variables
$\vec{z}$, and this may introduce fresh variables.  For such reasons, in the general bridge
theorem connecting the ES and infinite Beth properties, the logic $\,\vdash$ needs to be formulated
with a \emph{proper class} of variables.

On the other hand, many familiar algebraizable logics are finitary, with only countably many connectives,
and are formalized using a
countable set of variables---as nothing more is required for their axiomatization.  When working with such
logics, one would prefer a version of the infinite Beth property that also presupposes only
a \emph{set} of variables,
but is still provably equivalent to the unrestricted ES property for the algebraic counterpart.

It seems, however, that the published literature of abstract algebraic logic contains no such bridge theorem.
Analyzing the core proof in \cite{BH06}, we find that the gap would be filled by the following claim,
where $\mathfrak{m},\mathfrak{s}$ are cardinals:
\begin{quote}
If a prevariety $\cl{K}$ lacking the ES property has just $\mathfrak{s}$ operation symbols and an
axiomatization that uses only $\mathfrak{m}$ variables, then there is a non-surjective
$\cl{K}$--epimorphism $\sbA\mrig\sbB$, where $\sbB$ has at most $\mathfrak{m}+\mathfrak{s}+\aleph_0$
elements.
\end{quote}
This claim (generalized to first-order structures) is Theorem~\ref{thm:ESproperty}
of the pre\-sent paper.  One of its specializations (which follows also from a result of Bacsich \cite{Bac74}) says that
%
%
\begin{quote}
in a \emph{quasivariety} of countable type, if the epimorphisms into countable
structures are surjective, then so are all of the epimorphisms.
\end{quote}
%
%
We use these facts to obtain more manageable bridge
theorems (Theorem~\ref{bridge localized},
Corollaries~\ref{bridge localized cor}, \ref{bridge localized cor 2}),
connecting the unrestricted ES property with suitably \emph{localized}
infinite Beth properties for logics formalized with limited
variables and connectives.
The bridge theorems cater for all equivalential logics and specialize to the algebraizable ones.

The proof of Theorem~\ref{thm:ESproperty} rests on a general syntactic
characterization of the ranges of epimorphisms
(Theorem~\ref{thm:campercholi-rtl}), which widens the scope of Bacsich \cite[Thm.~1]{Bac74}
and
Campercholi \cite[Thm.~3]{Cam}.  (The earlier accounts apply only to classes
closed under ultraproducts.)

\section{Atomic Consequence}\label{sec:atomic consequence}

We work in the conservative extension NBG of ZFC (i.e., in the class theory of von Neumann, Bernays and G\"{o}del,
including the axiom of choice).

The \emph{cardinality of the signature} of a first-order language is the sum of the cardinalities
of its (disjoint) sets of operation and relation symbols.  The ranks of these symbols are assumed finite,
and nonzero in the case of relations.  Only first-order signatures will be considered.  A
signature is \emph{algebraic} if it has no relation symbols.

For a given signature, a structure denoted by $\spc{A}$
is assumed to have universe $A$ (a non-empty set)
and algebra reduct $\sbA=\langle A;O\rangle$
for a suitable set $O$ of operations on $A$, so that $\spc{A}=\langle A;O,R\rangle$ for a suitable set $R$
of relations on $A$.
The subalgebra of $\sbA$ and the substructure of $\spc{A}$ generated by a set $X\subseteq A$
have the same universe; the latter
is denoted by
$\Sg^\spc{A}(X)$.
A homomorphism $h\colon\spc{A}\mrig\spc{B}$
is understood to preserve the relations in $R$ (as well as the operations
in $O$), but it need not reflect the relations.  As usual, the \emph{kernel} $\{\langle a,a'\rangle\in A^2:h(a)=h(a')\}$
of $h$ is denoted by $\ker h$.

It is convenient here to have recourse to a fixed proper
class $\mathit{Var}$ of \emph{variables}.
For each sub\emph{set} $X$ of $\mathit{Var}$, we use $\sbT(X)$ to denote the absolutely free algebra (a.k.a.\
the term algebra) generated by $X$, with respect to the operation symbols of the
signature
under discussion.
Given a class $\Sigma\cup\{p\}$ of expressions over $\mathit{Var}$,
the set of variables occurring in $p$ shall be denoted by $\mathit{var}(p)$, while $\mathit{var}(\Sigma)\seteq\bigcup_{s\in\Sigma}\mathit{var}(s)$.

Recall that the \emph{atomic formulas} of a signature
are either \emph{equations} $\varphi\approx\psi$
or expressions $r(\varphi_1,\dots,\varphi_n)$, where $r$ is a relation symbol and
$\varphi,\psi,\varphi_1,\dots,\varphi_n$ are terms over $\mathit{Var}$.  We tacitly
identify $\varphi\approx\psi$ with the pair $\langle\varphi,\psi\rangle$.

If, upon \emph{introducing} a set $\Sigma$ of atomic formulas, we denote it as $\Sigma(\vec{x})$,
this signifies that the elements of $\mathit{var}(\Sigma)$ all belong to the sequence $\vec{x}$ of
\emph{distinct} variables.
(Strictly speaking, like any sequence, $\vec{x}$ is a function whose domain is some ordinal $\delta$,
so $\left|\vec{x}\right|=\left|\delta\right|$, which need not be finite.)
In this context,
given a structure $\spc{A}$ of matching signature, the notation $\spc{A}\models\Sigma(\vec{a})$
has the standard model-theoretic meaning, which entails in particular that $\vec{a}$ is a
$\de$--indexed sequence of not necessarily distinct elements of $A$ (briefly: $\vec{a}\in A$).
\begin{definition}\label{def:atomic consequence}
For any class
$\cl{K}$
of similar structures,
we use
the notation
${\Sigma(\vec{x})\models_\cl{K}p(\vec{x})}$ (or $\Sigma\models_\cl{K}p$)
to signify
that $\Sigma(\vec{x})\cup\{p(\vec{x})\}$ is a set
of atomic formulas in the signature of $\cl{K}$ and
\[
\textup{whenever $\spc{A}\in\cl{K}$ and $\spc{A}\models\Sigma(\vec{a})$ (where $\vec{a}\in A$),
then $\spc{A}\models p(\vec{a})$.}
\]
The displayed demand may be paraphrased a little more precisely as
\begin{tabbing}
\quad\quad \=\textup{for any $\spc{A}\in\cl{K}$
and any homomorphism $h\colon\sbT(\mathit{var}(\Sigma\cup\{p\}))\mrig\sbA$,}\\
\> \textup{if
$\spc{A}\models\Sigma(h[\vec{x}])$, then
$\spc{A}\models p(h[\vec{x}])$.}
\end{tabbing}
Here,
$h[\vec{x}]$ is the sequence obtained by applying $h$ to every
item of $\vec{x}$.  When $\Sigma\models_\cl{K}p$, we say that
`$\Sigma/p$ \emph{is validated by} (each
member of) $\cl{K}$'.  The ordered pair $\Sigma/p$ will then be referred to as an
\emph{(atomic) implication}; we call it
a \emph{quasi-atomic formula}
if $\Sigma$ is finite.\,\footnote{\,It could be rendered more suggestively as
$\left(\mbox{\normalsize \&}\, \Sigma\right)
\Longrightarrow p$, or as $\left(
\mbox{\normalsize \&}_{s\in\Sigma\,}\,
s\right)
\Longrightarrow
p$.}
\end{definition}

In this paper, we shall not need to deal with formulas more complex
than implications.
Any syntactic substitution instance $\Sigma(\vec{\varphi})$ of $\Sigma(\vec{x})$ shall have the form
$\textup{$\{s(h[\vec{x}]):s\in \Sigma\}$}$, where $h\colon\sbT(\mathit{var}(\Sigma))\mrig\sbT(Y)$
is a homomorphism, $Y$ being a subset of $\mathit{Var}$.

The class operator symbols $\Iop$, $\Hop$, $\Sop$, $\Pop$ and $\Puop$
stand, respectively, for closure under isomorphic
and homomorphic
(surjective) images, substructures,
(set-indexed) direct products and ultra\-products.

A class $\cl{K}$ of similar structures is said to be \emph{axiomatized} by a class $\Xi$ of
implications if $\cl{K}$ is the class of all structures validating all of the implications in $\Xi$.
In this case, $\cl{K}$ is a \emph{prevariety}, i.e., it is closed under $\Iop$, $\Sop$ and $\Pop$.

Conversely, every prevariety is axiomatized by a class $\Xi$ of implications; see Banaschewski and
Herrlich \cite{BH76}.  The claim that we cannot always find a \emph{set} to play the role of $\Xi$ (or
equivalently, of $\mathit{var}(\Xi)$) is consistent with NBG.
Its negation  (i.e.,
the claim that sets suffice) is consistent with NBG
if huge cardinals exist.
These facts were established by Ad\'{a}mek \cite{Adamek1990}.

A prevariety can be axiomatized by a class---w.l.o.g.\ a set---of
quasi-atomic [resp.\ atomic] formulas iff it is closed under $\Puop$ [resp.\ $\Hop$], in which
case we call it a \emph{quasivariety} [resp.\ a \emph{variety}], even if its members are not pure
algebras.  The reader may consult \cite[Ch.~2]{Gor98} for proofs of these well-known results (which
originate in \cite{Bir35,GL73,Mal66}).
\begin{definition}
\label{def:kappa-compact}
For any class $\cl{K}$ of similar structures and any infinite cardinal $\mathfrak{m}$, the relation $\models_\cl{K}$ is
said to be \emph{$\mathfrak{m}$--compact} provided that,
\begin{center}
whenever $\Sigma \models_{\cl{K}}
p$, then
$\Sigma' \models_{\cl{K}}
p$ for some $\Sigma' \subseteq \Sigma$ with $|\Sigma'| < \mathfrak{m}$.
\end{center}
We say that $\,\models_\cl{K}$ is \emph{finitary} if it is $\aleph_0$--compact.
\end{definition}
If $\cl{K}$ is closed
under $\mathbb{P}_\mathbb{U}$
(e.g., if $\cl{K}$ is a quasivariety), then $\,\models_\cl{K}$ is finitary.
The terminology of Definition~\ref{def:kappa-compact} applies, more
generally, to arbitrary relations from subsets of a class to elements of the same class.

\section{Epimorphisms}\label{sec:epimorphisms}

A morphism $h$ in a category $\cl{C}$ is called a ($\cl{C}$--)\,{\em epimorphism\/} provided that,
for any two $\cl{C}$--morphisms $f,g$ from the co-domain of $h$ to a single object,
\[
\textup{if $f\circ h=g\circ h$, then $f=g$.}
\]

We shall not distinguish notationally between a class $\cl{K}$ of similar (first-order) structures and
the concrete category of all homomorphisms between its members.  Clearly, within such a category, every surjective
homomorphism is an epimorphism.  If the converse holds, then $\cl{K}$ is said to have the {\em epimorphism
surjectivity property}, or briefly, the {\em ES property}.

A substructure $\spc{D}$ of a structure $\spc{B}\in\cl{K}$ is said to be
($\cl{K}$--)\,\emph{epic} in $\spc{B}$ if
each homomorphism from $\spc{B}$
to a member of $\cl{K}$ is determined by its restriction to $\spc{D}$.
(This is equivalent to the demand that the inclusion map $\spc{D}\mrig\spc{B}$ be an epimorphism, provided that $\cl{K}$
is closed under $\mathbb{S}$.)
Of course, a $\cl{K}$--morphism $h\colon\spc{A}\mrig\spc{B}$
is an epimorphism iff $h[\spc{A}]$ is an epic substructure of $\spc{B}$.  Therefore, $\cl{K}$ has the ES property iff each of its members
has no proper epic substructure.

The significance of epimorphisms for formal deductive systems (alluded to in the introduction) will
be elaborated in Section~\ref{sec:beth definability properties}.
For the moment, however, structures are our concern.  Already in the context of \emph{algebras},
we may recall that rings and distributive lattices each form varieties lacking the ES property.  This
reflects the absence of unary terms defining multiplicative inverses in rings, and complements in distributive
lattices, despite the uniqueness of those entities when they exist.

The connection between such `implicitly defined' constructs and epimorphisms
was remarked upon in the algebraic literature long ago (e.g., see
Freyd \cite[p.\,93]{Fre64} and Isbell \cite{Isb66}).  It was given a syntactically sharper characterization
by Bacsich \cite[Thm.~1]{Bac74} and by Campercholi \cite[Thm.~3]{Cam} (see
Remark~\ref{remk:dominions} below), but their accounts are
confined, respectively, to universal classes and to classes closed under ultra\-products.  With logical applications in mind,
we extend the result to arbitrary classes in the next theorem.


\begin{theorem}
\label{thm:campercholi-rtl}
Let\/ $\cl{K}$ be any class of similar structures, $\spc{A}$ a substructure of\/ $\spc{B} \in \cl{K}$ and\/ $Z\subseteq B
\ld A$\textup{,}
where\/
$\spc{B} = \Sg^{\spc{B}}(A\cup Z)$\textup{.}
Then the following conditions are equivalent.
\begin{enumerate}
\item\label{camp1}
$\spc{A}$ is\/ $\cl{K}$--epic in $\spc{B}$.

\item\label{camp2}
For each $b\in Z$\textup{,} there is a set\/ $\Sigma=\Sigma(\vec{x},\vec{z},v)$ of atomic formulas such that
$\spc{B}\models\Sigma(\vec{a},\vec{c},b)$ for suitable\/ $\vec{a}\in A$ and\/ $\vec{c}\in B$\textup{,} and
\begin{equation}\label{eq:camp}
\Sigma(\vec{x},\vec{z},v_1)\cup\Sigma(\vec{x},\vec{y},v_2)\models_\cl{K}v_1\approx v_2.
\end{equation}
\end{enumerate}
\noindent
In this case, for each $b\in Z$\textup{,} we can arrange that\/ $\left|\vec{z}\right|\leq\left|Z\right|$ and\/
$\vec{c}\in Z
$\textup{.}
\end{theorem}

\begin{proof}
(\ref{camp1})\,$\Rig$\,(\ref{camp2}):
Let $\vec{x}_A$ and $\vec{x}_Z$
be sequences of variables whose disjoint ranges,
$\{x_a:a\in A\}$ and $\{x_b:b\in Z\}$,
are bijective copies of $A$
and $Z$, respectively.
Let $\sbT=\sbT(\vec{x}_A,\vec{x}_Z)$ be the absolutely free algebra generated by the combined ranges of
$\vec{x}_A$ and $\vec{x}_Z$.
Let $h \colon \sbT
\mrig \sbB$ be the homomorphism such that $h(x_a)=a$ and $h(x_b)=b$ for all $a \in A$ and $b \in Z$.
Note that $h$ is surjective, as $A\cup Z$ generates $\sbB$.

Now fix $b \in Z$.
Let $\vec{x}_{Z\setminus\{b\}}$ denote the
subsequence of $\vec{x}_Z$ whose range omits $x_b$.
For notational convenience, assume that $\vec{x}_Z$ is ordered as $\vec{x}_{Z\setminus\{b\}},x_b$.

Let $\Sigma=\Sigma(\vec{x}_A,\vec{x}_{Z\setminus\{b\}},x_b)$ be the set of all atomic formulas
$r(\vec{x}_A,\vec{x}_Z)$
such that $\spc{B}\models r(h[\vec{x}_A],h[\vec{x}_Z])$.
Note that $\ker h\subseteq\Sigma$, as $\Sigma$ includes equations.

Because $\spc{B} \models \Sigma(h[\vec{x}_A],h[\vec{x}_{Z\setminus\{b\}}],h(x_b))$, where $h(x_b)=b$ and the items in $h[\vec{x}_A]$ and $h[x_{Z\setminus\{b\}}]$ belong to $A$ and $Z$ (respectively), it remains only to prove that
$\Sigma(\vec{x}_A,\vec{x}_{Z\setminus\{b\}},x_b)\cup
\Sigma(\vec{x}_A,\vec{y}_{Z\setminus\{b\}},y_b) \models_\cl{K} x_b \approx y_b$.

Let $\spc{C} \in \cl{K}$ and let $g_1,g_2\colon\sbT\mrig\sbC$ be homomorphisms that agree on $\vec{x}_A$, where
$\spc{C}\models r(g_1[\vec{x}_A],g_1[\vec{x}_Z])$ and
$\spc{C}\models r(g_2[\vec{x}_A],g_2[\vec{x}_Z])$
for all
$r(\vec{x}_A,\vec{x}_Z)\in \Sigma$.  Then $\ker h\subseteq\ker g_1\cap\ker g_2$.
We must show that $g_1(x_b)=g_2(x_b)$.

Let $j\in\{1,2\}$.  Because $h$ is surjective and $\ker h
\subseteq \ker g_j$,
the function $h(\varphi)\mapsto g_j(\varphi)$ is a well defined homomorphism $f_j\colon\sbB\mrig\sbC$.
In fact, $f_j$ is a homomorphism from $\spc{B}$ to $\spc{C}$, by the definitions of $\Sigma$ and $g_j$.
For each $a \in A$, we have $f_j(a) = f_j(h(x_a)) = g_j(x_a)$,
but $g_1$ and $g_2$ agree at $x_a$, so
$f_1|_A = f_2|_A$.  Then $f_1 = f_2$, since $\spc{C} \in \cl{K}$ and $\spc{A}$ is $\cl{K}$--epic in
$\spc{B}$. Therefore,
$\textup{$
g_1(x_b) = f_1(b) = f_2(b) = g_2(x_b)
$}$, as required.

(\ref{camp2})\,$\Rig$\,(\ref{camp1}):
Let
$g,h\colon\spc{B}\mrig\spc{C}\in\cl{K}$ be homomorphisms, with
$g|_A = h|_A$. We must show that $g = h$.  As $A\cup Z$ generates $\spc{B}$,
it suffices to prove that $g|_Z = h|_Z$.
Let $b \in Z$, and let $\Sigma$ and $\vec{a}\in A$ and $\vec{c}\in B$ be as in (\ref{camp2}).
From $\spc{B}\models\Sigma(\vec{a},\vec{c},b)$ we infer $\spc{C}\models\Sigma(g[\vec{a}],g[\vec{c}],g(b))$ and
$\spc{C}\models\Sigma(h[\vec{a}],h[\vec{c}],h(b))$.
But $g[\vec{a}]=h[\vec{a}]$, as $g|_A=h|_A$, so $g(b)=h(b)$, by (\ref{eq:camp}).
Thus, $g|_Z=h|_Z$.
\end{proof}

\begin{remark}\label{remk:finitarity}
In Theorem~\ref{thm:campercholi-rtl}(\ref{camp2}),
if $\,\models_{\cl{K}}$ is finitary (e.g., if $\cl{K}$ is closed under ultraproducts),
then each $\Sigma$ can be chosen finite.
\end{remark}

\begin{remark}\label{remk:dominions}
Given $\spc{A}\in\mathbb{S}(\spc{B})$, it is sometimes convenient to define the
$\cl{K}$--\emph{dominion} $\textup{dom}^\cl{K}_\spc{B}\spc{A}$ (of $\spc{A}$ in $\spc{B}$) as the set of all
$b\in B$ such that any two homomorphisms from $\spc{B}$ to a member of $\cl{K}$ will agree at $b$ if
they agree on $A$.  Then $\spc{A}$ is $\cl{K}$--epic in $\spc{B}$ iff
$\textup{dom}^\cl{K}_\spc{B}\spc{A}=B$.
In the above proof, if we choose $Z=B\ld A$, then the argument shows that, for any $b\in B$, we have
$b\in\textup{dom}^\cl{K}_\spc{B}\spc{A}$ iff there exist a set of atomic formulas $\Sigma(\vec{x},\vec{z},v)$
such that $\spc{B}\models\Sigma(\vec{a},\vec{c},b)$ for suitable\/ $\vec{a}\in A$ and\/ $\vec{c}\in B$\textup{,}
and (\ref{eq:camp}) holds.  Restricting to the case where $\cl{K}$ is closed under $\mathbb{P}_\mathbb{U}$
(whence each $\Sigma$ can be chosen finite, by Remark~\ref{remk:finitarity}), we obtain a more elementary proof of
the aforementioned result of Bacsich \cite[Thm.~1]{Bac74}, and likewise Campercholi
\cite[Thm.~3]{Cam}.\,\footnote{\,Dominions were introduced (for algebras) by Isbell \cite{Isb66}; also see
\cite{Bud04,Bud08,Hig88,Was01}.}
\end{remark}


\begin{remark}\label{remk:almost total}
In a structure $\spc{B}$, a substructure $\spc{A}$ is said to be \emph{almost total} if $\spc{B}=\Sg^\spc{B}(A\cup Z)$
for some \emph{finite} $Z\subseteq B$.
By Theorem~\ref{thm:campercholi-rtl},
the demand that an almost total substructure of $\spc{B}$ be $\cl{K}$--epic
is characterized by the existence of finitely many suitable implications (of possibly infinite length),
each having only finitely many variables in the role of $\vec{z}$.
We say that $\cl{K}$ has the \emph{weak ES property}
if no $\spc{B}\in\cl{K}$ has a proper
$\cl{K}$--epic almost total substructure.
It is pointed out in \cite[p.\,76]{BH06}
that the
meaning of this demand
would not change if, in the definition of `almost total', we required $\left|Z\right|=1$.
\end{remark}

\section{$\mathfrak{m}$--Prevarieties}\label{sec:generalized quasivarieties}

From now on, $\mathfrak{m}$ shall denote a fixed but arbitrary \emph{infinite} cardinal.  As usual,
$\mathfrak{m}^+$ stands for the cardinal successor of $\mathfrak{m}$.

As was mentioned in Section~\ref{sec:atomic consequence}, our ability to axiomatize arbitrary prevarieties
using only \emph{sets} of variables depends on the set theory in which we work
\cite{Adamek1990}.  This justifies our interest in the
following classes.

\begin{definition}
\label{def:GQV}
An \emph{$\mathfrak{m}$--prevariety}
is a class of structures axiomatized by
implications,
each of which is formulated in at most $\mathfrak{m}$ variables.
\end{definition}

Suppose $\Xi$ is a set of implications axiomatizing $\cl{K}$, where $\mathit{var}(\Xi)\subseteq Y$ and
$\textup{$\left|Y\right|\leq\mathfrak{m}$}$.  Then the set of atomic formulas over $Y$ has cardinality at most
$\mathfrak{n}\seteq\mathfrak{m}+\mathfrak{s}$, where $\mathfrak{s}$ is the cardinality of the signature.  Therefore,
each of the implications in $\Xi$ has at most $\mathfrak{n}$ atomic subformulas.
From this it follows easily
that $\cl{K}$ is \emph{closed under $\mathfrak{n}^+$--reduced products},
i.e., for any subfamily $\{\spc{A}_i:i\in I\}$ of $\cl{K}$ and any $\mathfrak{n}^+$--complete filter
$D$ over $I$, the reduced product $\prod_{i\in I}\spc{A}_i/D$ belongs to $\cl{K}$.  The demand that $D$
be $\mathfrak{n}^+$--\emph{complete} means that, whenever $E\subseteq D$ and $\left|E\right|\leq\mathfrak{n}$,
then $\bigcap E\in D$.  In summary:

\begin{lemma}\label{reduced products}
\textup{(cf.\ \cite[Prop.~2.3.19]{Gor98})}
\,Each\/ $\mathfrak{m}$--prevariety
is closed under\/ $\text{$(\mathfrak{m}+\mathfrak{s})^+$}$--reduced
products, where\/ $\mathfrak{s}$ is the cardinality of the signature.
\end{lemma}

\begin{theorem}
\label{thm:compactness}
Let\/ $\cl{K}$ be an\/ $\mathfrak{m}$--prevariety,
whose signature has cardinality $\mathfrak{s}$\textup{.} Then\/ $\,\models_{\cl{K}}$ is\/ $(\mathfrak{m}+\mathfrak{s})^+$--compact.
\end{theorem}

\begin{proof}
Let $\mathfrak{n}=\mathfrak{m}+\mathfrak{s}$ and suppose $\Sigma(\vec{x})\models_{\cl{K}}p(\vec{x})$,
where $\vec{x}=x_0,x_1,\dots$ is a sequence of (possibly more than $\mathfrak{n}$) variables.
Let $I=\{\Lambda\subseteq\Sigma:\left|\Lambda\right|\leq \mathfrak{n}\}$.
For each $\Lambda\in I$, let $I_\Lambda=\{\Gamma\in I:\Lambda\subseteq\Gamma\}$, so $I_\Lambda\neq\emptyset$
(as $\Lambda\in I_\Lambda$).  Define
\[
D=\{J\subseteq I:J\supseteq I_\Lambda\textup{ for some }\Lambda\in I\},
\]
so $\emptyset\notin D$, and $D$ is upward closed in the power set of $I$.
To see that $D$ is an $\mathfrak{n}^+$--complete filter over $I$, let $E\subseteq D$,
with $\left|E\right|\leq \mathfrak{n}$.  For each $J\in E$, choose $\Lambda_J\in I$ such that $J\supseteq I_{\Lambda_J}$.
Let $\Lambda=\bigcup_{J\in E}\Lambda_J$.  Because $\left|\Lambda_J\right|\leq \mathfrak{n}$ for all $J\in E$, we have
$\left|\Lambda\right|\leq\mathfrak{n}\bcdw\left|E\right|=\mathfrak{n}$, so $\Lambda\in I$.  Also,
$\bigcap E\supseteq\bigcap_{J\in E}I_{\Lambda_J}=I_\Lambda$, so $\bigcap E\in D$, as required.

Assume, with a view to contradiction, that for each $\Lambda\in I$, there exists $\spc{A}_\Lambda\in\cl{K}$
such that $\Lambda\not\models_{\{\spc{A}_\Lambda\}}p$, i.e., there exists
$\vec{a}_\Lambda=a^\Lambda_0,a^\Lambda_1,\ldots\in A_\Lambda$ such that $\spc{A}_\Lambda\models\Lambda(\vec{a}_\Lambda)$
but $\spc{A}_\Lambda\not\models p(\vec{a}_\Lambda)$.  Let $\spc{B}=\prod_{\Lambda\in I}\spc{A}_\Lambda$ and
$\spc{A}=\spc{B}/D$, so $\spc{A}\in\cl{K}$, by Lemma~\ref{reduced products}.  In particular, $\Sigma\models_{\{\spc{A}\}}p$.

Define $\vec{b}=b_0,b_1,\ldots\in B$
by $b_k(\Lambda)=a^\Lambda_k$, for each $k,\Lambda$.  If $s\in\Sigma$, then
\[
\{s\}\in I
\textup{ \ and \ }
I_{\{s\}}\subseteq \text{\textlbrackdbl$s(\vec{b})$\textrbrackdbl}\seteq
\{\Lambda\in I:\spc{A}_\Lambda\models s(b_0(\Lambda),b_1(\Lambda),\ldots)\},
\]
so $\text{\textlbrackdbl$s(\vec{b})$\textrbrackdbl}\in D$.
Thus, $\spc{A}\models\Sigma(\vec{b})$,
but $\text{\textlbrackdbl$p(\vec{b})$\textrbrackdbl}=\emptyset\notin D$, so $\spc{A}\not\models p(\vec{b})$.
This shows that $\Sigma\not\models_{\{\spc{A}\}}p$, a contradiction,
so $\Lambda\models_\cl{K}p$ for some $\Lambda\in I$.
\end{proof}

\section{ES Properties}\label{sec:ES properties}

A structure is said to be $\mathfrak{n}$--\emph{generated} (where $\mathfrak{n}$ is a cardinal)
if its algebra reduct has a generating subset with at most $\mathfrak{n}$ elements.
`Finitely generated' means $n$--generated for some $n\in\omega$.
(Recall that $\mathfrak{m}$ is infinite.)

\begin{theorem}
\label{thm:ESproperty}
Let\/ $\cl{K}$ be an\/ $\mathfrak{m}$--prevariety
whose signature has cardinality\/ $\mathfrak{s}$. Then\/ $\cl{K}$ has the ES property iff no structure in\/ $\cl{K}$ of cardinality at most\/ $\mathfrak{m}+\mathfrak{s}$
has a proper\/ $\cl{K}$--epic substructure.
\end{theorem}

\begin{proof}
Again, let $\mathfrak{n}=\mathfrak{m} + \mathfrak{s}$.
Suppose that $\cl{K}$ lacks the ES property, i.e., some $\spc{B} \in \cl{K}$ has a proper $\cl{K}$--epic substructure $\spc{A}$.
We must show that some
$\spc{C}\in\cl{K}$, with $\left|C\right|\leq\mathfrak{n}$, has a proper $\cl{K}$--epic substructure.

We shall define, recursively,
a denumerable sequence $\spc{C}_0,\spc{C}_1,\spc{C}_2,\dots$
of substructures of $\spc{B}$, where
$\spc{C}_i\in\mathbb{S}(
\spc{C}_{i+1})$ for each $i\in\omega$.

First, pick $b \in B \ld A$ and $a \in A$, and define $\spc{C}_0 = \Sg^{\spc{B}}\{a,b\}$, so $C_0\not\subseteq A$.

Now assume that $\spc{C}_i\in\mathbb{S}(\spc{B})$ has been defined, where $i\in\omega$, and that $C_i \not\subseteq A$.
Choose
$c \in C_i \ld A$.  As $c \in B \ld A$, Theorem~\ref{thm:campercholi-rtl}
shows that there exist a set of atomic formulas
$\Sigma(\vec{x},\vec{z},v)$
and elements $\vec{a}_c \in A$ and $\vec{d}_c\in B\ld A$, such that
$\spc{B} \models \Sigma(\vec{a}_c,\vec{d}_c,c)$ and
$\Sigma(\vec{x},\vec{z},v_1)
\cup
\Sigma(\vec{x},\vec{y},v_2)
\models_\cl{K}
v_1 \approx v_2$.
Moreover, $\,\models_{\cl{K}}$ is $\mathfrak{n}^+$--compact, by
Theorem~\ref{thm:compactness},
so we may assume that $\left|\Sigma\right|\leq\mathfrak{n}$, and hence
that
$\left|\mathit{var}(\Sigma)\right|\leq \mathfrak{n}$.
Consequently,
$|\vec{a}_c|,|\vec{d}_c| \leq \mathfrak{n}$.
Let $W_i$ be the union of (the ranges of) all the sequences $\vec{a}_c$ and
$\vec{d}_c$
such that $c\in C_i\ld A$, so
\begin{equation}\label{eq:wn}
\left|W_i\right|\leq
\mathfrak{n}\bcdw\left|C_i\right|.
\end{equation}
Define
$ \spc{C}_{i+1} = \Sg^{\spc{B}}(C_i \cup W_i)$.

Let $\spc{C}\in\mathbb{S}(\spc{B})$ be the (directed) union $\bigcup_{i\in\omega}\spc{C}_i$, so $\spc{C}\in\cl{K}$.
Now $D\seteq C\cap A$ is not empty,
as it includes $a$, so $D$
is the universe of a substructure $\spc{D}$ of $\spc{B}$.
Also, $\spc{D}$ is a proper substructure of $\spc{C}$,
as $b \in C \ld D$.

To see
that $\spc{D}$ is $\cl{K}$--epic in $\spc{C}$, let $c \in C \ld D$.
Then $c \in C_i \ld A$ for some $i \in \omega$.
Pick $\Sigma$, $\vec{a}_c$ and $\vec{d}_c$ as in the inductive step.
Because the substructure $\spc{C}_{i+1}$ of $\spc{C}$ includes $\vec{a}_c,\vec{d}_c,c$
and satisfies $\Sigma(\vec{a}_c,\vec{d}_c,c)$, the same is true of $\spc{C}$.
So, noting that $\vec{a}_c \in D$, we infer from
Theorem~\ref{thm:campercholi-rtl} that $\spc{D}$ is $\cl{K}$--epic in $\spc{C}$.

As $\mathfrak{s}+\aleph_0\leq\mathfrak{n}$,
the union $\spc{C}$ of the
family $\{\spc{C}_i:i\in\omega\}$ will have at most $\mathfrak{n}$ elements
if
every $\spc{C}_i$ is $\mathfrak{n}$--generated, which we
verify by induction.  Indeed,
$\spc{C}_0$ is $2$--generated, and if some $\spc{C}_i$ is $\mathfrak{n}$--generated, then so is
$\spc{C}_{i+1}$,
by (\ref{eq:wn}).
\end{proof}

The corollary below is due to Bacsich \cite{Bac74}.  (For \emph{varieties} of algebras, it
follows from an
earlier finding of Isbell \cite[Cor.~1.3]{Isb66}.)
\begin{corollary}\label{thm:ESpropertyQV}
\textup{(\cite[Thm.~2]{Bac74})}
Let\/ $\cl{K}$ be a quasivariety with a countable signature.  Then\/ $\cl{K}$ has the ES property if and only if no
countable
member of\/ $\cl{K}$ has a proper\/ $\cl{K}$--epic substructure.
\end{corollary}


\begin{proof}
Since $\cl{K}$ can be axiomatized by a set of finite implications, it is an $\aleph_0$--prevariety,
and the result follows from
Theorem~\ref{thm:ESproperty}.
\end{proof}


\begin{remark}\label{remk:no improvement}
An $\aleph_0$--prevariety
need not be a quasivariety, even if it has a variable-free axiomatization: see
\cite[p.\,45]{Adamek1990}.
Nevertheless, in
Corollary~\ref{thm:ESpropertyQV}, we cannot strengthen `countable' to `finitely generated'.
Indeed, a locally finite variety $\cl{K}$ of Brouwerian algebras and a proper $\cl{K}$--epic
subalgebra
of a denumerable member of $\cl{K}$ are exhibited in \cite[Sec.~6]{BMR17}, but
no finitely generated (i.e., finite) member of $\cl{K}$ has a proper $\cl{K}$--epic
subalgebra, because every variety of Brouwerian algebras has the \emph{weak} ES property (see Remark~\ref{remk:almost total}, \cite[Thm.~3.14]{BH06}
and
\cite{Kre60}).
\end{remark}
On the other hand, finitely generated structures do suffice, in quasivarieties, to test the weak ES property itself:

\begin{theorem}
\label{thm:weakES}
A quasivariety\/ $\cl{K}$ has the weak ES property iff no finitely generated member
of\/ $\cl{K}$ has a proper\/ $\cl{K}$--epic substructure.
\end{theorem}

\begin{proof}
Suppose that
some $\spc{B} \in \cl{K}$ has a proper $\cl{K}$--epic almost total substructure $\spc{A}$.
So, $\spc{B} = \Sg^\spc{B}(A\cup Z)$ for some finite non-empty set $Z \subseteq B \ld A$.
Let $b\in Z$.  Because $\cl{K}$ is a quasivariety, Remark~\ref{remk:finitarity}
shows that there exist a \emph{finite} set of atomic formulas $\Sigma_b(\vec{x},\vec{z},v)$ and elements
$\vec{a}_b \in A$ and $\vec{d}_b\in Z$ such that $\spc{B} \models \Sigma_b(\vec{a}_b,\vec{d}_b,b)$ and
$\Sigma_b(\vec{x},\vec{z},v_1) \cup \Sigma_b(\vec{x},\vec{y},v_2) \models_\cl{K} v_1 \approx v_2$.
As $\mathit{var}(\Sigma_b)$ is finite, the sequences $\vec{a}_b$ and $\vec{d}_b$ may be chosen finite.

Let $Y$ be the union of (the ranges of) all the sequences $\vec{a}_b$ such that $b\in Z$, so $Y$ is finite.
Let $\spc{A}' = \Sg^\spc{A} Y$ and $\spc{B}' = \Sg^\spc{B}(Y\cup Z)$.  Then
$\spc{A}'$ is a proper (almost
total) substructure of the finitely generated structure $\spc{B}'\in\cl{K}$.
For all $b\in Z$, we have
$\vec{a}_b,\vec{d}_b,b\in B'$ and $\spc{B}'\in\mathbb{S}(\spc{B})$, so
$\spc{B}'\models\Sigma_b(\vec{a}_b,\vec{d}_b,b)$.  Then, since $\vec{a}_b\in A$ for each $b\in Z$,
Theorem~\ref{thm:campercholi-rtl} shows that $\spc{A}'$ is $\cl{K}$--epic in $\spc{B}'$.
\end{proof}

\section{Equivalential and Algebraizable Logics}\label{sec:equivalential and algebraizable logics}

For a class $C$, we use $\mathcal{P}(C)$ to denote the class of all sub\emph{sets} of $C$.
\begin{definition}\label{def:logic set}
In a given algebraic signature, a \emph{deductive system} (briefly, a \emph{logic}) \emph{over} a set $X$
is a relation $\textup{$\,\vdash$}\subseteq\mathcal{P}(T(X))\times T(X)$ satisfying the demands below,
whenever $\Gamma\cup\Psi\cup\{\varphi\}\subseteq T(X)$:
\begin{enumerate}
\item\label{reflexivity}
$\Gamma\vdash\varphi$ for all $\varphi\in\Gamma$;
\item\label{transitivity}
if $\Gamma\vdash\psi$ for all $\psi\in\Psi$, and $\Psi\vdash\varphi$, then $\Gamma\vdash\varphi$;
\item\label{substitution invariance}
if $\Gamma\vdash\varphi$, then $h[\Gamma]\vdash h(\varphi)$ for all endomorphisms $h$ of $\sbT(X)$.
\setcounter{newexmp}{\value{enumi}}
\end{enumerate}
\end{definition}
Item (\ref{substitution invariance}) is called \emph{substitution-invariance}.
In this context, operation symbols
and elements of $\mathcal{P}(T(X))\times T(X)$
are usually called `connectives'
and `rules',
respectively.\,\footnote{\,Terms in $T(X)$ correspond intuitively to assertions (as in Boolean algebra), but
we resist the temptation to call them
`formulas', so as to prevent confusion with the first-order (e.g., atomic) formulas in the richer language
of the class $\cl{Mod}^*(\vdash)$, defined below.}

As the set of logics over $X$ is closed under arbitrary intersections, any subset $\Xi$
of $\mathcal{P}(T(X))\times T(X)$ generates such a logic (which is then also said to be \emph{axiomatized}
by $\Xi$).  A pair $\Gamma/\varphi$ in $\mathcal{P}(T(X))\times T(X)$ belongs to that logic iff $\varphi$
terminates
some (possibly infinite) sequence,
each item of which belongs to $\Gamma$ or is
$h(\psi)$ for some endomorphism $h$ of $\sbT(X)$ and some pair $\Psi/\psi$ from $\Xi$, where $h[\Psi]$
consists of previous items of the sequence.
(Here, $\Psi$ may be empty.)
This observation goes back, in principle, to \cite{LS58}.
\begin{definition}\label{def:logic class}
In a given algebraic signature,
a \emph{logic over} the proper class $\mathit{Var}$ is a family
$\textup{$\,\vdash$}=\{\vdash^X:X\in\mathcal{P}(\mathit{Var})\}$, where each $\,\vdash^X$ is a logic over $X$
(called a \emph{slice} of $\,\vdash$) and the following variant of (\ref{substitution invariance}) (called \emph{strong substitution-invariance}) holds:
\begin{enumerate}
\setcounter{enumi}{\value{newexmp}}
\item\label{substitution invariance2}
for any $X,Y\in\mathcal{P}(\mathit{Var})$,
if $\Gamma\textup{$\,\vdash^X$}\varphi$, then $h[\Gamma]\textup{$\,\vdash^Y$} h(\varphi)$ for all homomorphisms $h\colon\sbT(X)\mrig\sbT(Y)$.
\end{enumerate}
In this case, the notation $\Gamma\vdash\varphi$ signifies that $\Gamma\textup{$\,\vdash^X$}\varphi$ for some $X\in\mathcal{P}(\mathit{Var})$,
whence $\mathit{var}(\Gamma\cup\{\varphi\})\subseteq X$.  (Conversely, if
$\mathit{var}(\Gamma\cup\{\varphi\})\subseteq X\in\mathcal{P}(\mathit{Var})$ and $\Gamma\vdash\varphi$, then $\Gamma\textup{$\,\vdash^X$}\varphi$, by (\ref{substitution invariance2}).)
\end{definition}
Every logic $\,\vdash^*$ over a subset $X$ of the proper class $\mathit{Var}$ may be viewed as the $X$--slice of a logic $\,\vdash$ over
all of $\mathit{Var}$.  One such $\,\vdash$, which we label as \emph{induced by} $\,\vdash^*$, is defined by requiring that each of its
slices $\,\vdash^Y$ be
the logic over $Y$ generated by the set of all pairs $h[\Gamma]/h(\varphi)$ such
that $\Gamma\textup{$\,\vdash^*$}\varphi$ and $h\colon\sbT(X)\mrig\sbT(Y)$ is a homomorphism.  (This $\,\vdash$ satisfies
(\ref{substitution invariance2}),
by the syntactic characterization of $\,\vdash^Y$ preceding Definition~\ref{def:logic class}.)

Henceforth,
$\,\vdash$ is assumed to be a logic either over $\mathit{Var}$ or over an infinite subset of $\mathit{Var}$.
Note that
it is only for \emph{sets} $\Gamma\cup\{\varphi\}$ that the notation
$\Gamma\vdash\varphi$ is defined.
All claims in the present section
can be found in standard texts on abstract algebraic logic, e.g.,
\cite{BP89,Cze01,FJP03} and the recent \cite{Fon16}.  Their proofs are not affected by the extent
of the class of variables.  In fact, the class-versus-set distinction will be unimportant,
except in connection with Definitions~\ref{beth property classes} and \ref{beth property sets}
of Section~\ref{sec:beth definability properties}.

If $\Gamma\vdash\varphi$, then the pair $\Gamma/\varphi$ is called a \emph{derivable rule} of $\,\vdash$.
The expression
$\Gamma\vdash\Psi$ abbreviates `$\Gamma\vdash\xi$ for all $\xi\in\Psi$', while $\Gamma\dashv\vdash\Psi$
means `$\Gamma\vdash\Psi$ and $\Psi\vdash\Gamma$', and $\vdash\Psi$ stands for $\emptyset\vdash\Psi$.
(The same conventions will apply to relations of the form $\,\models_\cl{K}$ below.)

A ($\,\vdash$\,--)\,\emph{matrix} $\langle\sbA,F\rangle$
comprises an algebra $\sbA$ in the signature of $\,\vdash$ and a set $F\subseteq A$.
We regard it as a structure $\spc{A}$ for the signature whose operation symbols are the connectives
of $\,\vdash$ and whose sole relation symbol $r$ is unary, so that $\spc{A}\models r(a)$ iff
$a\in F$.  Intuitively, $r$ is a `truth' predicate.  The substructures $\langle\sbB,B\cap F\rangle$ ($\sbB\in\Sop(\sbA)$) of $\spc{A}$
are usually called \emph{submatrices}.  Similarly, \emph{matrix homomorphisms} are the homomorphisms
between matrices, considered as algebras with a distinguished unary relation.

Given a class $\cl{M}$ of
matrices, we abbreviate
$\{r(\gamma):\gamma\in\Gamma\}\models_\cl{M}r(\varphi)$ as
$\Gamma\models_\cl{M}\varphi$.  When this is true,
the rule
$\Gamma/\varphi$ is
said to be
\emph{validated by} (each member of) $\cl{M}$.
Abusing notation, we also use $\Gamma\models_\cl{M}\varphi\approx\psi$ to
abbreviate $\{r(\gamma):\gamma\in\Gamma\}\models_\cl{M}\varphi\approx\psi$
(where $\Gamma$ still consists of terms, not equations).

A matrix
$\langle\sbA,F\rangle$ is called a \emph{model of} $\,\vdash$ if it validates all the derivable rules of $\,\vdash$,
in which case $F$ is called a $\,\vdash$--\emph{filter} of $\sbA$.  The set $\mathit{Fi}_\vdash\sbA$ of all
$\,\vdash$--filters of $\sbA$ is closed under arbitrary intersections and is therefore the universe of a
complete lattice $\boldsymbol{\mathit{Fi}}_\vdash\sbA$, ordered by inclusion.

Given a matrix $\langle\sbA,F\rangle$, we denote by $\leibniz^\sbA F$ the largest
congruence $\theta$ of $\sbA$ for which $F$ is a union of $\theta$--classes (i.e., for which $b\in F$ whenever
both $\langle a,b\rangle\in\theta$ and $a\in F$).  This congruence always exists.  If $h\colon\sbB\mrig\sbA$ is a
homomorphism and $F$ is a $\,\vdash$--filter of $\sbA$, then $h^{-1}[F]\seteq\{b\in B:h(b)\in F\}$ is a
$\,\vdash$--filter of $\sbB$ and
\[
\leibniz^\sbB h^{-1}[F]\supseteq h^{-1}[\leibniz^\sbA F]\seteq\{\langle b,b'\rangle\in B^2:\langle h(b),h(b')\rangle
\in\leibniz^\sbA F\}.
\]
If, moreover, $h$ is surjective, then
\begin{equation}\label{inverse homs}
h^{-1}[\leibniz^\sbA F]\,=\,\leibniz^\sbB h^{-1}[F].
\end{equation}

The maps $F\mapsto\leibniz^\sbA F$ ($F\in \mathit{Fi}_\vdash\sbA$), taken over all algebras $\sbA$, constitute
the \emph{Leibniz operator} of $\,\vdash$.  This operator is not always isotone, i.e., from
$F,G\in \mathit{Fi}_\vdash\sbA$ and $F\subseteq G$, it need not follow that $\leibniz^\sbA F\subseteq\leibniz^\sbA G$.

A matrix $\langle\sbA,F\rangle$ is said to be
\emph{reduced} if $\leibniz^\sbA F=\textup{id}_A\seteq\{\langle a,a\rangle:a\in A\}$.
The derivable rules of $\,\vdash$ are exactly the pairs
$\Gamma/\varphi$ validated by the class $\cl{Mod}^*(\vdash)$ of all reduced matrix models of $\,\vdash$
(i.e., $\Gamma\vdash\varphi$ iff ${\Gamma\models_{\cl{Mod}^*(\vdash)}\varphi}$).
We treat $\cl{Mod}^*(\vdash)$ as a concrete category, equipped with
all matrix homo\-morphisms between
its members.

\begin{theorem}\label{thm:equivalential}
The following conditions on\/ $\,\vdash$ are
equivalent.
\begin{enumerate}
\item\label{equivalential 0}
$\cl{Mod}^*(\vdash)$ is a prevariety.

\item\label{equivalential 1}
The Leibniz operator of\/ $\,\vdash$ is isotone (for all algebras) and\/ \textup{(\ref{inverse homs})} holds for\/ \emph{all}
homomorphisms $h\colon\sbB\mrig\sbA$ and all\/ $\,\vdash$--filters $F$ of\/ $\sbA$\textup{.}

\item\label{equivalential 2}
There exists a set $\Delta$ of binary terms such that, for any matrix model\/ $\langle\sbA,F\rangle$
of\/ $\,\vdash$\textup{,}
we have\/ $\leibniz^\sbA F=\{\langle a,b\rangle\in A^2:\Delta^\sbA(a,b)\subseteq F\}$\textup{.}
\item\label{equivalential 3}
There exists a set $\Delta$ of binary terms such that
\begin{align*}
& \quad\quad\quad \vdash\Delta(x,x);\\
& \quad\quad\quad \{x\}\cup\Delta(x,y)\vdash y;\\
&
\quad\quad\quad \Delta(x_1,y_1)\cup\,\dots\,\cup\Delta(x_n,y_n)
\vdash\Delta(\varphi(x_1,\dots,x_n),\varphi(y_1,\dots,y_n)),
\end{align*}
for every connective\/ $\varphi$ of\/ $\,\vdash$\textup{,} where $n$ is the rank of\/ $\varphi$\textup{.}
\end{enumerate}
A set\/ $\Delta$ of binary terms witnesses\/ \textup{(\ref{equivalential 2})} iff it witnesses\/
\textup{(\ref{equivalential 3}).}  In that case, the third
demand in\/ \textup{(\ref{equivalential 3})}
generalizes from connectives $\varphi$ to arbitrary terms.
\end{theorem}
We say that $\,\vdash$ is \emph{equivalential} if the conditions in Theorem~\ref{thm:equivalential} hold.
The elements of the set $\Delta$ in (\ref{equivalential 2}) or (\ref{equivalential 3})
are then called \emph{equivalence formulas} for $\,\vdash$,
and they are unique in the sense that $\Delta(x,y)\dashv\vdash\Delta'(x,y)$ for any other such set $\Delta'$.
In this case,
$\Delta(x,y)\models_{\cl{Mod}^*(\vdash)}x\approx y$, by (\ref{equivalential 2}).  For the roots
of Theorem~\ref{thm:equivalential}, see \cite[pp.\,222--3]{Woj88}, as well as \cite{BP92,Cze81,Cze01,Her97,PW74}.

\begin{remark}\label{axiomatizations 1}
Suppose $\Delta$ is a set of equivalence formulas for $\,\vdash$, with $r$ as above.
Then $\cl{Mod}^*(\vdash)$ is clearly axiomatized
by the implications
\[
\{r(\gamma):\gamma\in\Gamma\}\,/\,r(\varphi)
\]
corresponding
to the rules $\Gamma/\varphi$ in any given axiomatization $\Xi$ of $\,\vdash$, together with the postulate
\[
\{r(\rho(x,y)):\rho\in\Delta\}\,/\,x\approx y.
\]
Thus,
if $\mathit{var}(\Xi)$ is a set, then $\cl{Mod}^*(\vdash)$ is a
$(\left|\mathit{var}(\Xi)\right|+\aleph_0)$--prevariety.
In general, $\cl{Mod}^*(\vdash)$ is a quasivariety iff $\,\vdash$ is
\emph{finitely equivalential} (i.e., equipped with a finite
set of equivalence formulas) and finitary \cite{Cze81,BP92}.
\end{remark}

We define $\cl{Alg}^*(\vdash)=\{\sbA: \langle\sbA,F\rangle\in\cl{Mod}^*(\vdash)\textup{ for some }F\}$.

Given a class $\cl{K}\cup\{\sbA\}$ of similar algebras, let $\mathit{Con}_{\cl{K}}\alg{A}$
denote the set of all $\cl{K}$--\emph{congruences} of
$\sbA$, i.e., all congruences $\theta$ such that
$\alg{A}/\theta \in \cl{K}$.  If $\cl{K}$ is a prevariety, then
$\mathit{Con}_{\cl{K}}\alg{A}$
is closed under arbitrary
intersections and
is therefore the universe of
a complete lattice,
${\boldsymbol{\mathit{Con}}}_{\cl{K}}\alg{A}$,
ordered by
inclusion.

\pagebreak[2]

\begin{theorem}\label{thm:algebraizable}
The following conditions on\/ $\,\vdash$ are
equivalent.
\begin{enumerate}
\item\label{algebraizable 1}
$\,\vdash$ is equivalential and its reduced matrix models are
determined by their algebra reducts, i.e.,
whenever\/ $\langle\sbA,F\rangle,\langle\sbA,G\rangle\in\cl{Mod}^*(\vdash)$\textup{,} then\/ $F=G$.
\item\label{algebraizable 2}
$\cl{Alg}^*(\vdash)$ is a prevariety and, for each algebra $\sbA$\textup{,} the map\/ $F\mapsto\leibniz^\sbA F$
defines a lattice isomorphism from\/ $\boldsymbol{\mathit{Fi}}_\vdash\sbA$ onto\/
$\boldsymbol{\mathit{Con}}_{\cl{Alg}^*(\vdash)}\sbA$\textup{.}

\item\label{algebraizable 3}
There exist a class\/ $\cl{K}$ of algebras, a set\/ $\{\langle\de_i,\varepsilon_i\rangle:i\in I\}$ of pairs of unary terms
and a set\/ $\Delta$
of binary terms
such that, for any set\/ $\Gamma\cup\{\varphi\}$ of terms,
\begin{align*}
& \quad\quad\quad\; \Gamma\vdash\varphi \textup{ \,iff\, } \{\de_i(\gamma)\approx\varepsilon_i(\gamma):\gamma\in \Gamma,\, i\in I\}\models_{\cl{K}}
\{\de_i(\varphi)\approx\varepsilon_i(\varphi):i\in I\};\\
& \quad\quad\quad\quad\quad\;\; \{\de_i(\rho(x,y))\approx\varepsilon_i(\rho(x,y)):i\in I,\,\rho\in \Delta\}\, =\!\mid\models_\cl{K} \,x\approx y.
\end{align*}
\end{enumerate}
In this case,
$\Delta$
is a set of equivalence formulas for\/ $\,\vdash$\textup{,} and\/
$\cl{Alg}^*(\vdash)$ is the unique prevariety\/ $\cl{K}$ of algebras for which\/
\textup{(\ref{algebraizable 3})} holds.
\end{theorem}
We say that $\,\vdash$ is \emph{algebraizable} and, more explicitly, that $\cl{Alg}^*(\vdash)$ \emph{algebraizes} $\,\vdash$,
if the conditions of Theorem~\ref{thm:algebraizable} hold.
The pairs in \textup{(\ref{algebraizable 3})} are then unique
in the sense that
\[
\{\de_i(x)\approx\varepsilon_i(x):i\in I\}\,=\!\mid\models_{\cl{Alg}^*(\vdash)}\{\de_j'(x)\approx\varepsilon_j'(x):j\in J\}
\]
for any other such set $\{\langle\de_j',\varepsilon_j'\rangle:j\in J\}$\textup{.}
In this case, when ${\langle\sbA,F\rangle\in{\cl{Mod}^*(\vdash)}}$,
then
$F=\{a\in A:\de_i^\sbA(a)=\varepsilon_i^\sbA(a) \text{ for all }i\in I\}$.  The concrete categories $\cl{Mod^*}(\vdash)$ and
$\cl{Alg}^*(\vdash)$ are therefore isomorphic when $\,\vdash$ is algebraizable.

The original definition of algebraizability
is due to Blok and
Pigozzi \cite{BP89}.  Its scope was widened in \cite{BJ06,Cze01,Her97} and adapted to
logics over proper classes in \cite{CP99}.  For the origins of Theorem~\ref{thm:algebraizable},
see \cite{BP89,BP92} also.

\begin{remark}\label{axiomatizations 2}
When the conditions of Theorem~\ref{thm:algebraizable} hold, then
$\,\vdash$ is axiomatized by the postulates captured in Theorem~\ref{thm:equivalential}(\ref{equivalential 3}),
in the relation
\[
x\dashv\vdash\mbox{$\bigcup$}_{i\in I}\Delta(\de_i(x),\varepsilon_i(x))
\]
and in
the rules $\bigcup_{\langle\xi,\eta\rangle\in\Sigma}\Delta(\xi,\eta)\vdash\Delta(\varphi,\psi)$
corresponding to the equational implications $\Sigma/\varphi\approx\psi$ belonging to any axiomatization of
$\cl{Alg}^*(\vdash)$.  In this case, therefore, if $\cl{Alg}^*(\vdash)$ is an $\mathfrak{m}$--prevariety,
then $\,\vdash$ can be axiomatized using at most $\mathfrak{m}$ variables.
\end{remark}

\section{Beth Definability Properties}\label{sec:beth definability properties}

\begin{definition}\label{beth property classes}
(\cite{BH06}) A logic $\,\vdash$ over the proper class $\mathit{Var}$ is said to have the \emph{(deductive) infinite Beth
(definability) property} if the following holds for all disjoint subsets $X,Z$ of $\mathit{Var}$, with $T(X)\neq\emptyset$,
and all $\Gamma\subseteq T(X\cup Z)$: if,
\begin{tabbing}
\quad\quad \= \textup{for each $z\in Z$ and each homomorphism $h\colon \sbT(X\cup Z)\mrig \sbT(Y)$,}\\
\> \textup{with $Y\in\mathcal{P}(\mathit{Var})$, such that
$h(x)=x$ for all $x\in X$, we have}\\
\> \textup{$\Gamma\cup h[\Gamma]\models_{\cl{Mod}^*(\vdash)} z\approx h(z)$,}
\end{tabbing}
then, for each $z\in Z$, there exists $\varphi_z\in T(X)$ such that $\Gamma\models_{\cl{Mod}^*(\vdash)}z\approx\varphi_z$.
\end{definition}

\begin{theorem}\label{thm:BH06}
\textup{(\cite[Thm.~3.12]{BH06})}
\,Let\/ $\,\vdash$ be an equivalential logic over a proper class.
Then $\,\vdash$ has the
infinite Beth property iff, in the prevariety\/ $\cl{Mod}^*(\,\vdash)$\textup{,} all epimorphisms are
surjective.
\end{theorem}

We have not found the following definition in the published literature.

\begin{definition}\label{beth property sets}
Let $\,\vdash$ be a logic over an infinite set $V$.
We shall say that $\,\vdash$ has the ($V$--) \emph{localized infinite Beth property} provided that the
following is true for all disjoint subsets $X,Z$ of $V$ and all $\Gamma\subseteq T(X\cup Z)$, such that
$T(X)\neq\emptyset$ and $\left|V\ld(X\cup Z)\right|\geq \left|Z\right|+\aleph_{0\,}$: \,if,
\begin{tabbing}
\quad\quad \= \textup{for each $z\in Z$ and each endomorphism $h$ of $\sbT(V)$, such that}\\
\> \textup{$h(x)=x$ for all $x\in X$, we have $\Gamma\cup h[\Gamma]\models_{\cl{Mod}^*(\vdash)} z\approx h(z)$,}
\end{tabbing}
then, for each $z\in Z$, there exists $\varphi_z\in T(X)$ such that $\Gamma\models_{\cl{Mod}^*(\vdash)}z\approx\varphi_z$.
\end{definition}

The displayed assumptions in Definitions~\ref{beth property classes} and \ref{beth property sets}
will both be pronounced as `$\Gamma$ \emph{defines $Z$ implicitly in terms of $X$} in $\,\vdash$'.
(There is no ambiguity, since the two possibilities for $\,\vdash$ are mutually exclusive.)
The term $\varphi_z$ in the conclusion is called an \emph{explicit definition} of $z$ in terms of $X$,
with respect to $\Gamma$, in $\,\vdash$.

If $\Delta$ is a set of equivalence formulas for $\,\vdash$, then in Definitions~\ref{beth property classes} and
\ref{beth property sets},
we may replace
$\Gamma\cup h[\Gamma]\models_{\cl{Mod}^*(\vdash)}z\approx h(z)$ and
$\Gamma\models_{\cl{Mod}^*(\vdash)}z\approx\varphi_z$
by the intrinsic (but equivalent) respective demands
\begin{equation*}
\Gamma\cup h[\Gamma]\vdash \Delta(z,h(z)) \textup{ \,and\, } \Gamma\vdash \Delta(z,\varphi_z).
\end{equation*}

\begin{lemma}\label{bp implies bpv}
Let\, $\,\vdash$ be an equivalential logic
over the proper class $\mathit{Var}$\textup{,}
and let\/
$V\in\mathcal{P}(\mathit{Var})$ be infinite.
If\/ $\,\vdash$ has the infinite Beth property, then $\,\vdash^V$ has the localized infinite Beth property.
\end{lemma}
\begin{proof}
It is given that $\,\vdash$ has a set $\Delta$ of equivalence formulas.
Let $X$, $Z$ and $\Gamma$ be as in Definition~\ref{beth property sets}, and $h\colon\sbT(X\cup Z)\mrig\sbT(X\cup\ov{Z})$
as in Definition~\ref{beth property classes}, where $\ov{Z}\seteq\mathit{var}(h[Z])$.
Set $W=V\ld(X\cup Z)$.  As $V$ is infinite and $\Delta$ consists of binary terms,
we may assume that $\Delta\subseteq T(V)$.  In view of Theorem~\ref{thm:equivalential}(\ref{equivalential 3}),
$\Delta$ is also a set of equivalence formulas
for $\,\vdash^V$, so
we need only show that $\Gamma\cup h[\Gamma]\vdash\Delta(z,h(z))$ for all $z\in Z$.

As
$\left|\ov{Z}
\right|\leq\left|Z\right|+\aleph_0\leq\left|W
\right|$,
there are homomorphisms
$f\colon\sbT(\ov{Z})\mrig\sbT(W)$ and
$g\colon\sbT(W)\mrig\sbT(\ov{Z})$, with $g\circ f=\textup{id}_{T(\ov{Z})}$.
Let $q$ be the endomorphism of $\sbT(V)$ that agrees with $f\circ h$ on $Z$ and that fixes all elements of $V\ld Z$.
Assuming that $\Gamma$ defines $Z$ implicitly in terms of $X$ in $\,\vdash^V$, we infer that
$\Gamma\cup q[\Gamma]\,\textup{$\,\vdash^V$}\,\Delta(z,q(z))$ for all $z\in Z$.
Then, by Definition~\ref{def:logic class}(\ref{substitution invariance2}),
\begin{equation}\label{eq:pq}
p[\Gamma]\cup pq[\Gamma]\vdash p[\Delta(z,q(z))] \textup{ \,for all $z\in Z$},
\end{equation}
where $p\colon\sbT(V)\mrig\sbT(X\cup Z\cup\ov{Z})$ is the
homomorphism that agrees with $g$ on $W$, while fixing all elements of $X\cup Z$.  Now (\ref{eq:pq}) simplifies to
\[
\Gamma\cup h[\Gamma]\vdash\Delta(z,h(z)) \textup{ \,for all $z\in Z$},
\]
by the definitions of $p$ and $q$, and since $g\circ f=\textup{id}_{T(\ov{Z})}$.
\end{proof}

\begin{theorem}\label{thm:pre bridge}
Let\/ $\,\vdash$ be an equivalential logic over an infinite set\/ $V$\textup{,} where\/ $\,\vdash$ has at most\/ $\left|V\right|$
connectives and has
the localized infinite Beth
property.  Then no member of\/ $\cl{Mod}^*(\vdash)$
with at most\/ $\left|V\right|$ elements
has a proper\/
${\cl{Mod}^*(\vdash)}$--epic submatrix.
\end{theorem}
\begin{proof}
As $V$ is infinite, it can be partitioned into sets $V'$ and $V\ld V'$, both of cardinality
$\left|V\right|$.  Let\/ $\mathcal{B}=\langle \sbB,F_B\rangle\in\cl{Mod}^*(\vdash)$, with $\left|B\right|\leq\left|V\right|$,
and consider a $\cl{Mod}^*(\vdash)$--epic submatrix
$\mathcal{A}=\langle\sbA,A\cap F_B\rangle$ of $\mathcal{B}$.  As $\left|B\right|\leq
\left|V'\right|$, there
is a surjection $V'\mrig B$, which may be extended to a homomorphism $h\colon\sbT(V')\mrig\sbB$.  Assume, with a
view to contradiction, that $A\neq B$, and let
\[
\textup{$X=V'\cap h^{-1}[A]$
\,and\, $Z=V'\cap h^{-1}[B\ld A]$,}
\]
so $X\cup Z=V'$ and $X\cap Z=\emptyset\neq T(X)$.  Then
$\Gamma\seteq h^{-1}[F_B]\in\mathit{Fi}_{\vdash\,}\sbT(V')$.

As $V\ld(X\cup Z)=V\ld V'$ has the same (infinite) cardinality as $V'$, which contains $Z$,
we have $\left|V\ld(X\cup Z)\right|\geq\left|Z\right|+\aleph_0$.
Therefore, because $\mathcal{A}$ is $\cl{Mod}^*(\vdash)$--epic in $\mathcal{B}$, it can be shown, just as in the proof of
\cite[Thm.~3.12]{BH06}, that for any endomorphism $k$ of $\sbT(V)$ which
fixes all elements of $X$, and any homomorphism $g\colon\sbT(V)\mrig\sbC$, where $\langle\sbC,F_C\rangle\in\cl{Mod}^*(\vdash)$,
\[
\textup{if
$g[\Gamma\cup k[\Gamma]]\subseteq F_C$, then $g(z)=g(k(z))$ for all $z\in Z$.}
\]
That is,
$\Gamma$ defines $Z$ implicitly in terms of $X$ in $\,\vdash$.

Pick $b\in B\ld A$.  As $h|_{V'}\colon V'\mrig B$ is surjective, we have $b=h(z)$ for some $z\in Z$.
By the localized infinite Beth property, $\Gamma\models_{\cl{Mod}^*(\vdash)}z\approx\varphi_z$ for some
$\varphi_z=\varphi_z(\vec{x})\in T(X)$.
As $\mathcal{B}\in\cl{Mod}^*(\vdash)$ and $h[\Gamma]\subseteq F_B$, it follows that
\[
b=h(z)=h(\varphi_z)=\varphi^\sbB(h[\vec{x}]).
\]
Now $h[\vec{x}]$ consists of elements of $A$, and $\sbA\in\mathbb{S}(\sbB)$, so $b\in A$.  This
contradiction shows that $A=B$, as required.
\end{proof}

We can now prove the following bridge theorem, in which the need to
consider proper classes is eliminated for symbolically limited logics.

\begin{theorem}\label{bridge localized}
Let\/ $\,\vdash$ be an equivalential logic over an infinite subset\/ $V$ of the proper class\/ $\mathit{Var}$\textup{.}
Let\/ $\mathfrak{s}$
be the cardinality of the signature of\/ $\,\vdash$\textup{,} and assume that\/ $\,\vdash$ has an axiomatization that uses at most\/
$\mathfrak{m}$ variables, where\/
$\mathfrak{m}+\mathfrak{s}\leq\left|V\right|$\textup{.}
Then the following conditions are equivalent.
\begin{enumerate}
\item\label{bridge 1}
$\,\vdash$ has the
localized infinite Beth property with respect to\/ $V$\textup{.}

\item\label{bridge 2}
No member of\/
$\cl{Mod}^*(\vdash)$ with at most\/ $\left|V\right|$ elements has a proper\/ $\cl{Mod}^*(\vdash)$--epic submatrix.

\item\label{bridge 3}
All epimorphisms in\/ $\cl{Mod}^*(\vdash)$ are surjective.

\item\label{bridge 4}
The logic over\/ $\mathit{Var}$ induced by\/ $\,\vdash$
has the infinite Beth property.
\end{enumerate}
\end{theorem}
\begin{proof}
(\ref{bridge 1})$\;\Rig\;$(\ref{bridge 2}) instantiates Theorem~\ref{thm:pre bridge}.

(\ref{bridge 2})$\;\Rig\;$(\ref{bridge 3}): By Remark~\ref{axiomatizations 1},
$\cl{Mod}^*(\vdash)$ is a\/ $\left|V\right|$--prevariety,
so the present implication follows from Theorem~\ref{thm:ESproperty}.

(\ref{bridge 3})$\;\Rig\;$(\ref{bridge 4}): As $\,\vdash$ and the induced logic $\,\vdash'$ have the same
matrix models, $\cl{Mod}^*(\vdash)=\cl{Mod}^*(\vdash')$.  The implication therefore follows from Theorem~\ref{thm:BH06}.\nopagebreak

(\ref{bridge 4})$\;\Rig\;$(\ref{bridge 1}) instantiates Lemma~\ref{bp implies bpv}.
\end{proof}

\begin{corollary}\label{bridge localized cor}
Theorem~\textup{\ref{bridge localized}} remains true for algebraizable logics\/ $\,\vdash$ if we replace
the terms `submatrix' and\/ \textup{`}$\cl{Mod}^*(\vdash)$\textup{'} by\/ `subalgebra' and\/
\textup{`}$\cl{Alg}^*(\vdash)$\textup{',}
respectively.
\end{corollary}
\begin{proof}
Apply to Theorem~\ref{bridge localized}
the category isomorphism between
$\cl{Mod}^*(\vdash)$ and $\cl{Alg}^*(\vdash)$ that is guaranteed by the algebraizability of $\,\vdash$.
\end{proof}

\begin{corollary}\label{bridge localized cor 2}
Let\/ $\,\vdash$ be a logic with a countable signature, over a denumerable set\/ $V$ of
variables.  Consider the following conditions.
\begin{enumerate}
\item\label{localized 1}
$\,\vdash$ has the localized infinite Beth property with respect to\/ $V$.
\item\label{localized 2}
$\cl{Mod}^*(\vdash)$ has the ES property.
\item\label{localized 3}
$\cl{Alg}^*(\vdash)$ has the ES property.
\end{enumerate}

If\/ $\,\vdash$ is finitary and equivalential, then\/ \textup{(\ref{localized 1})} and\/ \textup{(\ref{localized 2})} are
equivalent.

If\/ $\,\vdash$ is algebraized by a quasivariety, then\/ \textup{(\ref{localized 1})} and\/
\textup{(\ref{localized 3})} are equivalent.
\end{corollary}
\begin{proof}
A finitary equivalential logic
with a
countable signature satisfies
the
initial
hypotheses of Theorem~\ref{bridge localized}
when $\mathfrak{m}=\aleph_0$.  For the second claim, use
Remark~\ref{axiomatizations 2} (with $\mathfrak{m}=\aleph_0$) and Corollary~\ref{bridge localized cor}.
\end{proof}

Even for algebraizable logics, the finiteness assumptions in the two claims of Corollary~\ref{bridge localized cor 2} are mutually independent: see \cite{Her96} and
\cite{Raf10}, respectively.

The \emph{finite Beth property} is defined like the infinite one, except that the set $Z$ in its definition is
required to be finite.  An equivalential
logic $\,\vdash$ over a proper class has this property iff $\cl{Mod}^*(\vdash)$ has the weak ES property
\cite[Thm.~3.14, Cor.~3.15]{BH06}.

Let us also define the ($V$--) \emph{localized finite Beth property} like its infinite analogue,
but stipulating that $Z$ be finite and substituting `$V\ld X$ is infinite' for
`$\left|V\ld(X\cup Z)\right|\geq\left|Z\right|+\aleph_0$'.

Theorem~\ref{bridge localized} and its corollaries have analogues for these properties.  We state only one,
wherein the cardinality of the signature plays no role.

\begin{theorem}\label{thm:bridge localized finite}
Let\/ $\,\vdash$ be a logic over a denumerable subset\/ $V$ of the proper class\/
$\mathit{Var}$\textup{.}
Consider the following conditions.
\begin{enumerate}
\item\label{localized finite 1}
$\,\vdash$ has the localized finite Beth property with respect to\/ $V$.
\item\label{localized finite 2}
No finitely generated member of\/ $\cl{Mod}^*(\vdash)$ has a proper\/ $\cl{Mod}^*(\vdash)$--epic
submatrix.
\item\label{localized finite 3}
$\cl{Mod}^*(\vdash)$ has the weak ES property.
\item\label{localized finite 4}
The logic over\/ $\mathit{Var}$ induced by\/ $\,\vdash$ has the finite Beth property.
\item\label{localized finite 5}
No finitely generated member of\/ $\cl{Alg}^*(\vdash)$ has a proper\/ $\cl{Alg}^*(\vdash)$--epic
subalgebra.
\item\label{localized finite 6}
$\cl{Alg}^*(\vdash)$ has the weak ES property.
\end{enumerate}

If\/ $\,\vdash$ is finitary and finitely equivalential, then\/ \textup{(\ref{localized finite 1})--(\ref{localized finite 4})} are
equivalent.

If\/ $\,\vdash$ is algebraized by a quasivariety, then\/
\textup{(\ref{localized finite 1}),\,(\ref{localized finite 4}),\,(\ref{localized finite 5})} and\/ \textup{(\ref{localized finite 6})}
are equivalent (even if\/ $\,\vdash$ is not finitary).
\end{theorem}
\begin{proof}
(\ref{localized finite 1})\,$\Rig$\,(\ref{localized finite 2}):  Adapt the proof of Theorem~\ref{thm:pre bridge},
arranging that $h$ maps $V'$ onto a finite generating set for $\sbB$.  At the end, instead
of ${b=h(z)}$, we have $b=\psi^\sbB(h[\vec{x}],h[\vec{z}])$ for a suitable term $\psi$ and finite sequences $\vec{x}\in X$
and $\vec{z}\in Z$.  Apply the original argument to the items in $\vec{z}$.

Both
(\ref{localized finite 2})\,$\Rig$\,(\ref{localized finite 3}) and (\ref{localized finite 5})\,$\Rig$\,(\ref{localized finite 6})
instantiate Theorem~\ref{thm:weakES}, because $\cl{Mod}^*(\vdash)$ is a quasivariety even in the former case
(see Remark~\ref{axiomatizations 1}).

Otherwise, the proof is like that of Theorem~\ref{bridge localized} and
its corollaries.
\end{proof}

For equivalential logics, the meaning of the finite Beth property
is not affected if we stipulate in its definition that $Z$ be a
singleton.  This is deduced from relevant
bridge theorems in \cite[Cor.~3.15]{BH06}, and the argument
applies equally to the localized finite Beth property.  Therefore,
by the proof in \cite{Kre60}, all six conditions of
Theorem~\ref{thm:bridge localized finite} hold when $\,\vdash$ is
an axiomatic extension of intuitionistic propositional logic, or of
its positive fragment (cf.\ Remark~\ref{remk:no improvement}).

For simplicity, we have confined the above discussion to sentential logics, but the results of this section extend
straightforwardly to the `$k$--deductive systems' of \cite{BH06,BP92} and to Gentzen systems as formulated, for
instance, in \cite{Raf06}.

\end{document}